\documentclass[a4paper]{amsart}
\usepackage{amsmath, amssymb, amsthm}
\usepackage{mathtools}
\usepackage{tikz-cd}
\usepackage{enumitem}
\usepackage{hyperref}
\usepackage{comment}
\usepackage{enumitem}

\usepackage{amscd}

\usepackage{latexsym}

\usepackage[all]{xy}
\usepackage[utf8]{inputenc}

\usepackage{amsfonts}

\usepackage{booktabs}
\usepackage{multirow}
\usepackage{colortbl}

\usepackage{graphicx}

\usepackage[table]{}

\theoremstyle{definition}
\newtheorem{construction}{Construction}[section]
\newtheorem{definition}[construction]{Definition}
\newtheorem{example}[construction]{Example}

\theoremstyle{plain}
\newtheorem{proposition}[construction]{Proposition}
\newtheorem{lemma}[construction]{Lemma}
\newtheorem{theorem}[construction]{Theorem}
\newtheorem{corollary}[construction]{Corollary}
\newtheorem{openp}[construction]{Open Problems}

\theoremstyle{remark}

\newtheorem{remark}[construction]{Remark}

\newcommand{\Hom}[3]{\operatorname{Hom}_{#1}(#2,#3)}
\newcommand{\End}[2]{\operatorname{End}_{#1}(#2)}
\newcommand{\Ext}[4]{\operatorname{Ext}^{#1}_{#2}(#3,#4)}

\newcommand{\rfmod}[1]{\mbox{\rm{mod}--}{#1}}
\newcommand{\rmod}[1]{\mbox{\rm{Mod}--}{#1}}

\newcommand{\lmod}[1]{{#1}\mbox{--\rm{Mod}}}

\newcommand{\pd}[2]{\mbox{\rm{proj.dim\,}}_{#1}{#2}}

\newcommand{\id}[2]{\mbox{\rm{inj.dim\,}}_{#1}{#2}}

\DeclareMathOperator{\coker}{coker}
\DeclareMathOperator{\im}{Im}
\DeclareMathOperator{\Ker}{Ker}

\DeclareMathOperator{\Add}{Add}
\DeclareMathOperator{\Prod}{Prod}
\DeclareMathOperator{\lgd}{l.gl.dim}
\DeclareMathOperator{\rgd}{r.gl.dim}
\DeclareMathOperator{\lfd}{l.fin.dim}
\DeclareMathOperator{\rfd}{r.fin.dim}
\DeclareMathOperator{\lFd}{l.Fin.dim}
\DeclareMathOperator{\rFd}{r.Fin.dim}

\newcommand{\La}{\Lambda}
\newcommand{\catC}{\mathcal{C}}
\newcommand{\catD}{\mathcal{D}}

\newcommand{\smat}[4]{\begin{psmallmatrix} #1 & #2 \\ #3 & #4 \end{psmallmatrix}}

\begin{document}

\title{Tilting modules for the Cummings construction}

\author{Jan Trlifaj and Mykyta Dubov}
\address{Charles University, Faculty of Mathematics
and Physics, Department of Algebra \\
Sokolovsk\'{a} 83, 186 75 Prague 8, Czech Republic}
\email{trlifaj@karlin.mff.cuni.cz}
\email{mykyta.dubov356@student.cuni.cz}

\date{\today}
\begin{abstract} Finiteness of the right little finitistic dimension of a finite dimensional algebra $\La$ is known to be equivalent to existence of a (possibly infinite dimensional) tilting right $\La$-module $T_f$ whose tilting class is $\{ T_f \}^{\perp_\infty} = (\mathcal P ^{< \infty})^{\perp}$, \cite{AT}. We use this equivalence to interpret the recent surprising results of Cummings \cite{C} concerning the asymmetry of left and right finitistic dimensions of the triangular matrix algebras $\tilde{A}$ built from arbitrary basic finite dimensional algebras $A$. In particular, we determine the structure of the tilting right $\tilde{A}$-module $T_f$ in the case when the algebra $A$ has finite right global dimension.
\end{abstract}

\thanks{Research supported by grant GA\v CR 26-22734S}


\maketitle

\section{Introduction}\label{intro}

Cummings has recently proved a surprising result concerning the left-right asymmetry of the finitistic dimensions of finite-dimensional algebras: given an algebraically closed field $k$ and any basic finite dimensional $k$-algebra $A$ of right little finitistic dimension $\rfd {A} = n$, there exists a finite-dimensional $k$-algebra $\tilde{A}$ with $\rfd {\tilde{A}} = m$ such that 
$n \leq m \leq n + 1$, but the left little finitistic dimension $\lfd {\tilde{A}} = 0$, \cite{C}. 

By \cite{AT}, the fact that $\rfd {R} = n < \infty$ for a right noetherian ring $R$ is equivalent to the existence of a possibly infinitely generated $n$-tilting (right $R$-) module $T_{f}$ whose tilting class is $\{ T_{f} \}^{\perp_\infty} = (\mathcal P ^{< \infty})^{\perp}$. Here, $\mathcal P ^{< \infty}$ denotes the class of all finitely generated modules of finite projective dimension, and for a class of modules $\mathcal C$, $\mathcal C ^{\perp} := \{ N \in \rmod R \mid \Ext 1RCN = 0 \hbox{ for all } C \in \mathcal C \}$, and $\mathcal C ^{\perp_\infty} := \{ N \in \rmod R \mid \Ext iRCN = 0 \hbox{ for all } 1 \leq i < \infty \hbox{ and } C \in \mathcal C \}$.

While $T_{f}$ is unique up to equivalence of tilting modules, it may be hard to determine its structure even for finite dimensional algebras of right little finitistic dimension $1$, as witnessed by the case of the IST-algebra investigated in \cite{St}. Indeed, $T_{f}$ can be taken finitely generated only if the category $\mathcal P ^{< \infty}$ is contravariantly finite in $\rfmod R$, cf.\ \cite[Theorem 4.2]{AT}. However, if $R$ has finite right global dimension, or if $R$ is Iwanaga-Gorenstein, then there is a simple choice for $T_f$, namely the minimal injective cogenerator of $\rmod R$. 

\medskip
Cummings' result implies that if there is a finite dimensional $k$-algebra $A$ that fails the Second Finitistic Dimension Conjecture (i.e., $\rfd A = \infty$), then there is one that fails it too, but its left little finitistic dimension equals $0$. Though the Second Finitistic Dimension Conjecture has been proven in many particular cases, it remains open in general. In contrast, the First Finitistic Dimension Conjecture, claiming that $\rfd A = \rFd A$ for each finite dimensional $k$-algebra $A$, is known to fail. In fact, the difference $\rFd A - \rfd A$ can be arbitrarily big, see \cite{H} and \cite{Sm}.  

\medskip
In the present paper, we prove several general results about modules of finite projective dimension over the Cummings algebra $\tilde{A}$. In Theorem \ref{generalrel}, we then relate, for each $n \geq 1$, the $n$-tilting classes in $\rmod {\tilde{A}}$ to ($n-1$)-tilting classes in $\rmod A$. In the final section, we concentrate on the case when $A$ has right global dimension $n < \infty$. First, we show that in that case, the algebra $\tilde{A}$ has infinite right global dimension, but both its right big, and its right little, finitistic dimensions equal $n+1$. Our main result, Theorem \ref{main}, then determines the structure of the ($n+1$)-tilting $\tilde{A}$-module $\tilde{T}_{f}$.

\section{Preliminaries}\label{prelim}

\subsection{Triangular matrix rings and their modules}\label{trmr}

The Cummings construction in \cite{C} is a particular instance of the construction of a triangular matrix ring studied in detail in \cite[III.2]{ARS}, which in turn is a special case of a trivial extension investigated in \cite[Section 4]{FGR}. \footnote{Cummings' result has recently been extended to appropriate trivial extensions in \cite{K}.}

We now briefly recall several results relevant for us from \cite[III.2]{ARS}. Notice that while \cite[III.2]{ARS} considers left modules, we will consider both left and right modules, so we will need some of the results in two, mutually dual, forms. 

\medskip
In what follows, $k$ will be a field, $A$ and $B$ be finite-dimensional $k$-algebras, and ${}_B M_A$ a $(B,A)$-bimodule. We will stick to the convention that homomorphisms are written and composed from the opposite side of the scalars. So homomorphism of (right $R$-) modules will be acting from the left, and vice versa.

\begin{definition}\label{trimr} \rm The \emph{triangular matrix $k$-algebra} $\La$ associated to $A$, $B$, and $M$ is the set
\[
\La \;=\; \begin{pmatrix} A & 0 \\ M & B \end{pmatrix}
\;=\; \left\{ \begin{pmatrix} a & 0 \\ m & b \end{pmatrix} : a \in A,\; m \in M,\; b \in B \right\},
\]
equipped with componentwise addition, and the multiplication
\[
\begin{pmatrix} a & 0 \\ m & b \end{pmatrix}
\begin{pmatrix} a' & 0 \\ m' & b' \end{pmatrix}
=
\begin{pmatrix} aa' & 0 \\ ma' + bm' & bb' \end{pmatrix}.
\]
The identity element of $\La$ is $\smat{1_A}{0}{0}{1_B}$.
\end{definition}

Let $\varepsilon_1 = \smat{1_A}{0}{0}{0}$ and $\varepsilon_2 = \smat{0}{0}{0}{1_B}$. Clearly, $\{ \varepsilon_1, \varepsilon_2 \}$ is a complete set of orthogonal idempotents of $\La$ such that there are $k$-algebra isomorphisms $A \cong \varepsilon_1\Lambda\varepsilon_1$ and $B \cong \varepsilon_2\Lambda\varepsilon_2$. 

\medskip
We now recall a convenient presentation of the module categories $\rmod {\La}$ and $\lmod {\La}$ from \cite[III.2]{ARS}: 

\begin{definition}\label{catC} \rm Define $\catC_\La$ as the category whose
\begin{itemize}[leftmargin=2em]
\item objects are the triples $(N, P, f)$ with $N \in \rmod A$, $P \in \rmod B$, and $f: P \otimes_B M \to N$ a homomorphism of right $A$-modules,
\item morphisms $(N, P, f) \to (N', P', f')$ are the pairs $(\alpha, \beta)$ where $\alpha: N \to N'$ is an $A$-homomorphism, $\beta: P \to P'$ is a $B$-homomorphism, and the following diagram in $\rmod A$ is commutative 
\[
\begin{tikzcd}
P \otimes_B M \arrow[r, "\beta \otimes_B id"] \arrow[d, "f"] & P' \otimes_B M \arrow[d, "f'"] \\
N \arrow[r, "\alpha"] & N'
\end{tikzcd}
\]
\end{itemize}
Addition and composition of morphisms is defined componentwise; the identity morphism on $(N, P, f)$ is $(id_N, id_P)$. 
\end{definition}

\begin{proposition}\label{requiv} (dual of \cite[Proposition III.2.2]{ARS}) Consider the functor $F_r: \catC_\La \to \rmod {\La}$ defined 
\begin{itemize}[leftmargin=2em]
\item on objects by $F_r(N, P, f) = N \oplus P$, an abelian group with the right $\La$-action
\[
(n, p) \cdot \begin{pmatrix} a & 0 \\ m & b \end{pmatrix} = \bigl(na + f(p \otimes m),\; pb\bigr),
\]
\item on morphisms by $F_r(\alpha, \beta) = \alpha \oplus \beta: N \oplus P \to N' \oplus P'$.
\end{itemize}
Then $F_r$ is an equivalence of categories. The inverse equivalence $G_r: \rmod {\La} \to \catC_\La$ of $F$ is given
\begin{itemize}[leftmargin=2em]
\item \emph on objects by $G_r(U) = (U\varepsilon_1,\; U\varepsilon_2,\; f_U)$ where
$f_U: U\varepsilon_2 \otimes_B M \to U\varepsilon_1$ is defined by
\[
f_U(u\varepsilon_2 \otimes m) \;=\; u\varepsilon_2 \cdot \smat{0}{0}{m}{0} \;\in\; U\varepsilon_1,
\]
\item on morphisms by $G_r(g) = (g \restriction {U\varepsilon_1},\; g \restriction {U\varepsilon_2})$.
\end{itemize}
\end{proposition}

For left $\La$-modules, we have

\begin{definition}\label{catD} \rm Define $\catD_\La$ as the category whose
\begin{itemize}[leftmargin=2em]
\item objects are triples $(N, P, f)$ with $N \in \lmod A$, $P \in \lmod B$, and $f: M \otimes_A N \to P$ a morphism of left $B$-modules,
\item morphisms $(N, P, f) \to (N', P', f')$ are pairs $(\alpha, \beta)$ where $\alpha: N \to N'$ is a left $A$-homomorphism, $\beta: P \to P'$ is a left $B$-homomorphism, and the following diagram in $\lmod B$ is commutative 
\[
\begin{tikzcd}
M \otimes_A N \arrow[r, "id \otimes_A \alpha"] \arrow[d, "f"] & M \otimes_A N^\prime \arrow[d, "f'"] \\
P \arrow[r, "\beta"] & P'
\end{tikzcd}
\]
\end{itemize}
\end{definition}

\begin{proposition}\label{lequiv} \cite[Proposition III.2.2]{ARS} Consider the functor $F_l: \catD_\La \to \lmod {\La}$ defined 
\begin{itemize}[leftmargin=2em]
\item on objects by $F_l(N, P, f) = N \oplus P$ as an abelian group, with the left $\La$-action
\[
\begin{pmatrix} a & 0 \\ m & b \end{pmatrix} \cdot (n, p) = \bigl(an, f(m \otimes p) + bp\bigr),
\]
\item on morphisms by $F_l(\alpha, \beta) = \alpha \oplus \beta: N \oplus P \to N' \oplus P'$.
\end{itemize}
Then $F_l$ is an equivalence of categories. 
\end{proposition}

Notice that the categories $\rmod A$ and $\rmod B$ can be embedded into the category $\catC_\La$ in several ways: 

First, the functor $H_A: \rmod A \to \catC_\La$ defined by $H(N) = (N, 0, 0)$ and $H_A(\alpha) = (\alpha, 0)$ is a fully faithful exact embedding. In this way, $\rmod A$ may be viewed as a full subcategory of $\catC_\La$. 

A different fully faithful embedding is provided by the functor $K_A: \rmod A \to \catC_\La$ defined by $K_A(N) = (N,\; \Hom{A}{M}{N},\; \phi_N)$ and $K_A(\alpha) = (\alpha,\; \Hom {A}{M}{\alpha})$ where $\phi_N: \Hom{A}{M}{N} \otimes_B M \to N$ is the evaluation map $\phi_N(\varphi \otimes m) = \varphi(m)$. 

There is also the fully faithful exact embedding $H_B: \rmod B \to \catC_\La$ defined by $H_B(P) = (0, P, 0)$ and $H_B(\beta) = (0,\beta)$, and the fully faithful embedding $K_B: \rmod B \to \catC_\La$ defined by $K_B(P) = (P \otimes_B M, P, id)$ and $K_B(\beta) = (\beta \otimes_B id,\beta)$.

Analogous functors are available for embeddings of the categories $\lmod A$ and $\lmod B$ into the category $\catD_\La$.  

\medskip
We recall from \cite[Proposition III.2.5]{ARS} the structure of simple, projective and injective modules in $\lmod {\La}$ and $\rmod {\La}$. Using the category equivalences from Propositions \ref{requiv} and \ref{lequiv}, one obtains the following characterizations:

\begin{proposition}\label{simp-proj-inj} 
\begin{itemize}[leftmargin=2em]
\item[{(i)}] The simple objects in $\catC_\La$ are exactly the triples of the form $H_A(S) = (S,0,0)$ and $H_B(T) = (0,T,0)$ where $S$ is a simple right $A$-module and $T$ a simple right $B$-module.       
\item[{(ii)}]The indecomposable projective objects in $\catC_\La$ are exactly the triples of the form $H_A(N) = (N,0,0)$ and $K_B(P) = (P \otimes_B M,P,id)$ where $N$ is an indecomposable projective right $A$-module and $P$ an indecomposable projective right $B$-module. 
\item[{(iii)}] The indecomposable injective objects in $\catC_\La$ are exactly the triples of the form $H_B(P) = (0,P,0)$ and $K_A(N) = (N,\Hom{A}{M}{N},\phi_N)$ where $N$ is an indecomposable injective right $A$-module and $P$ an indecomposable injective right $B$-module.	
\item[{(iv)}] The simple objects in $\catD_\La$ are exactly the triples of the form $(S,0,0)$ and $(0,T,0)$ where $S$ is a simple left $A$-module and $T$ a simple left $B$-module.
\item[{(v)}] The indecomposable projective objects in $\catD_\La$ are exactly the triples of the form $(N,M \otimes_A N,id)$ and $(0,P,0)$ where $N$ is an indecomposable projective left $A$-module and $P$ an indecomposable projective left $B$-module. 
\item[{(vi)}] The indecomposable injective objects in $\catD_\La$ are exactly the triples of the form $(\Hom {B}{M}{P},P,\psi_P)$ and $(N,0,0)$ where $N$ is an indecomposable injective left $A$-module, $P$ an indecomposable injective left $B$-module, and $\psi_P(m \otimes_A g) = (m)g$ for all $m \in M$ and $g \in \Hom {B}{M}{P}$.
\end{itemize}
\end{proposition}

\begin{remark}\label{pd-id} \rm Since the functor $H_A$ is exact and takes projective right $A$-modules to projective right $\La$-modules, it maps projective resolutions to projective resolutions. Hence $\pd{A}{N} = \pd{\La}{(N,0,0)}$ for each module $N \in \rmod A$. In particular, $\rgd {A} \leq \rgd {\tilde{A}}$, $\rFd {A} \leq \rFd {\tilde{A}}$, and $\rfd {A} \leq \rfd {\tilde{A}}$.
 
Similarly, the functor $H_B$ yields $\id{B}{P} = \id{\La}{(0,P,0)}$ for each module $P \in \rmod B$. Analogous statements hold for left $A$- and left $B$-modules. 
\end{remark}

\subsection{Infinite-dimensional tilting modules}\label{tilt}

Tilting theory originated in the setting of finitely generated modules over finite dimensional algebras. Its extension to general modules over general rings made it possible to encompass the setting of commutative rings, where all finitely generated titling modules are projective. As mentioned in the Introduction, even for finite dimensional algebras, it is important to admit infinite dimensional tilting modules, as the tilting module $T_{f}$ such that $\{ T_{f} \}^{\perp_\infty} = (\mathcal P ^{< \infty})^{\perp}$ can be taken finitely generated only if the category $\mathcal P ^{< \infty}$ is contravariantly finite in $\rfmod R$. 

\medskip
We now recall the basic facts of infinite dimensional tilting theory needed in the sequel. For more details and proofs, we refer to \cite[Part III]{GT}.

In the following, $R$ denotes a ring. A module $M \in \rmod R$ is \emph{strongly finitely presented} in case $M$ has a projective resolution consisting of finitely generated (projective) modules. Of course, if $R$ is right noetherian, then strongly finitely presented = finitely generated. For a class of modules $\mathcal C$, we let $\mathcal C ^{< \infty}$ denote the class of all strongly finitely presented modules in $\mathcal C$.

For each $n < \omega$, $\mathcal{P}_n$ denotes the class of all modules $M$ of projective dimension $\pd{R}{M} \leq n$. Moreover, we let $\mathcal{P} = \bigcup_{n < \omega} {\mathcal P _n}$ be the class of all modules of finite projective dimension.

\begin{definition}\label{deftilt} \rm
Let $T \in \rmod R$. Then $T$ is a \emph{tilting module} if it satisfies the following conditions:
\begin{itemize}[leftmargin=3em]
\item[\textup{(T1)}] $T \in \mathcal P$. 
\item[\textup{(T2)}] $\Ext {i}{R}{T}{T^{(\kappa)}} = 0$ for all $i \geq 1$ and all cardinals $\kappa$.
\item[\textup{(T3)}] There exist an $r \geq 0$ and an exact sequence $0 \to R \to T_0 \to \dots \to T_r \to 0$ such that $T_i \in \Add(T)$ for all $i \leq r$.
\end{itemize}
Here, $\Add(T)$ denotes the class of all direct summands of (possibly infinite) direct sums of copies of $T$. If $T \in \mathcal P _n$, then $T$ is called \emph{$n$-tilting}.

Two titing modules $T$ and $T^\prime$ are \emph{equivalent} in case $\Add(T) = \Add(T^\prime)$.
\end{definition}

\begin{definition}\label{defcotpair} \rm
Let $\mathcal{C} \subseteq \rmod R$. For $i \geq 1$, we define:
\begin{itemize}
\item[{(i)}]  $\mathcal{C}^{\perp} := \{ N \in \rmod R \mid \Ext 1RMN = 0 \ \forall\, M \in \mathcal{A} \}$,
\item[{(ii)}]  $\mathcal{C}^{\perp_{\infty}} := \{ N \in \rmod R \mid \Ext {i}{R}{C}{N} = 0 \ \forall i \geq 1, \, \forall\, C \in \mathcal{C} \}$,
\end{itemize}
The classes ${}^\perp \mathcal{C}$ and ${}^{\perp_{\infty}} \mathcal C$ are defined dually.  

A pair of classes of modules $(\mathcal{A}, \mathcal{B})$ is a \emph{cotorsion pair} if $\mathcal{A} = {}^{\perp}\mathcal{B}$ and $\mathcal{B} = \mathcal{A}^{\perp}$.
The cotorsion pair $(\mathcal{A}, \mathcal{B})$ is \emph{hereditary}, if $\Ext iRAB = 0$ for all $i \geq 1$, $A \in \mathcal A$ and $B \in \mathcal B$. For example, for each $n \geq 0$, $(\mathcal P_n,(\mathcal P_n)^\perp)$ is a hereditary cotorsion pair, cf.\ \cite[8.10]{GT}.

Let $\mathcal C$ be a class of modules. A hereditary cotorsion pair $(\mathcal{A}, \mathcal{B})$ is said to be \emph{generated} by $\mathcal C$ in case $\mathcal B = \mathcal C^{\perp_\infty}$.  

A hereditary cotorsion pair $\mathfrak C = (\mathcal{A}, \mathcal{B})$ is \emph{($n$-) tilting}, if $\mathfrak C$ is generated by $\mathcal C = \{ T \}$ where $T$ is an ($n$-) tilting module. In this case $\Add(T) = \mathcal{A} \cap \mathcal{B}$, \cite[13.10]{GT}.
\end{definition}

\begin{proposition}\label{tiltclasses} \cite[13.13]{GT} 
Let $T$ be a tilting module and $(\mathcal{A}_T, \mathcal{B}_T)$ be the tilting cotorsion pair generated by $T$. Then:
\begin{enumerate}[label=\textup{(\roman*)}]
\item $\mathcal{B}_T$ is called the \emph{tilting class} of $T$. It coincides with the class all \emph{$\Add(T)$-resolved} modules, i.e.\ the modules $M \in \rmod R$ admitting a (possibly infinite) resolution
\[
\cdots \longrightarrow T_2 \longrightarrow T_1
\longrightarrow T_0 \longrightarrow M \longrightarrow 0
\]
where $T_i \in \Add(T)$ for each $i \geq 0$.
\item The class $\mathcal{A}_T$ coincides with the class of all modules $M \in \rmod R$ admitting a finite coresolution
\[
0 \longrightarrow M \longrightarrow T^0 \longrightarrow T^1
\longrightarrow \cdots \longrightarrow T^r \longrightarrow 0
\]
where $r < \infty$ and $T^i \in \Add(T)$ for each $i \geq 0$.
\end{enumerate}
\end{proposition}

\begin{theorem}\label{chartilt} \cite[13.20]{GT}
Let $n < \omega$. Let $\mathfrak{C} = (\mathcal{A}, \mathcal{B})$ be a cotorsion pair.
The following assertions are equivalent:
\begin{enumerate}
\item $\mathfrak{C}$ is an $n$-tilting cotorsion pair;
\item $\mathfrak{C}$ is a hereditary cotorsion pair such that $\mathcal{A} \subseteq \mathcal{P}_n$, and $\mathcal{B}$ is closed under direct sums.
\end{enumerate}
\end{theorem}

\begin{lemma}\label{tiltshift} \cite[13.45]{GT}
Let $R$ be a ring and $n \geq 1$. Let $T$ be a tilting module of projective dimension $n$, and $\mathfrak{C} = (\mathcal{A}^\prime, \mathcal{B}^\prime)$ be the hereditary cotorsion pair generated by $\mathcal C = \{ \Omega(T) \}$, where $\Omega(T)$ is the first syzygy of $T$. Then $\mathfrak{C}$ is an ($n-1$)-tilting cotorsion pair.
\end{lemma}

\begin{theorem}\label{fintype} \cite[13.46]{GT}
Let $n \geq 0$ and $T$ be an ($n$-) tilting module. Let $\mathfrak C = (\mathcal{A}_T, \mathcal{B}_T)$ be the hereditary cotorsion pair generated by $T$. 

Then $T$ and $\mathcal{B}_T$ are \emph{of finite type}, that is, there exists a set $\mathcal S$ consisting of strongly finitely presented modules of finite projective dimension ($\leq n$) which generates $\mathfrak C$. For example, one can take $\mathcal S = \mathcal{A}_T^{<\omega}$.    
\end{theorem}

\begin{proposition}\label{A-description} \cite[13.11]{GT}
Let $n \geq 0$, $T$ be an $n$-tilting module, and $(\mathcal{A}_T, \mathcal{B}_T)$ be the hereditary cotorsion pair generated by $T$. Then
\[
\mathcal{A}_T \;=\; \mathcal{P} \cap {}^{\perp_\infty}\Add(T).
\]
In particular, if $\rgd(R) < \infty$ then $\mathcal{A}_T = {}^{\perp_\infty}\Add(T)$.
\end{proposition}

\subsection{Tilting modules and finitistic dimension conjectures}\label{findimt}  

We now recall the relation between (infinitely generated) tilting modules and finitistic dimensions of rings from \cite{AT} 
(see also \cite[\S 17]{GT}). 

\begin{definition}\label{defdim}  Let $R$ be a ring. The right global dimension of $R$, denoted by $\rgd (R)$, is the supremum of the projective dimensions of all modules $M \in \rmod R$. The right big (little) finitistic dimension of $R$, denoted by $\rFd (R)$ ($\rfd (R)$), is the supremum of the projective dimensions of all (finitely generated) modules which have finite projective dimension. 

Similarly, the left global, big and little finitistic dimensions of $R$, denoted by $\lgd (R)$, $\lFd (R)$ and $\lfd(R)$, respectively, are defined using left $R$-modules.  
\end{definition}

Clearly always $\rfd (R) \leq \rFd (R) \leq \rgd (R)$. Since $\rgd (R)$ is the supremum of the projective dimensions of all cyclic modules, if $\rgd (R) < \infty$, then $\rgd (R) = \rFd (R) = \rfd (R)$. If $R$ is right semiartinian, then $\rgd (R)$ is just the supremum of the projective dimensions of all simple modules. 

The key connection with infinite dimensional tilting theory comes from the following theorem from \cite{AT} (see also \cite[17.10]{GT}):

\begin{theorem}\label{findim} Let $R$ be a right noetherian ring. Then $\rfd {R} = n < \infty$, if and only if there exists an $n$-tilting module $T_f$ whose tilting class is $\{ T_f \}^{\perp_\infty} = (\mathcal{P}^{<\infty})^\perp$. 
\end{theorem}

Though $T_f$ is uniquely determined up to equivalence of tilting modules, its computation may be rather involved even if $R$ is a finite dimensional algebra with $\rfd (R) = 1$, see \cite{St}. There are however cases when $T_f$ is easy to compute: this happens when $R$ is right noetherian of finite right global dimension, or when $R$ is Iwanaga-Gorenstein:

\begin{proposition}\label{easy} Let $R$ be a right noetherian ring of finite right injective dimension such that each indecomposable injective module has projective dimension $\leq n$. Let $T$ be the direct sum of a representative set of all indecomposable injective modules. Then $T = T_f$ is the $n$-tilting module such that $\{ T \}^{\perp_\infty} = (\mathcal{P}^{<\infty})^\perp$. Moreover, $\rfd (R) = \rFd (R) = n$. 
\end{proposition}

\begin{proof} Since $T$ is injective, $\pd {R}{T} \leq n$ by the assumptions. Moreover, $R$ is right noetherian, so Condition (T2) from Definition \ref{deftilt} holds for $T$, and $\Add (T)$ is the class of all injective modules. Since $R$ has finite injective dimension, also Condition (T3) holds for $T$. So $T$ is an $n$-tilting module. 

By Proposition \ref{tiltclasses}(ii), $\mathcal A_{T}$ is the class of all modules of finite injective dimension. In particular, all projective modules, and hence all modules of finite projective dimension, have also finite injective dimension. Thus $\mathcal P \subseteq \mathcal A_{T}$. However, $\mathcal A_{T} \subseteq \mathcal P_n$ by Theorem \ref{chartilt}, so $\mathcal A_{T} = \mathcal P _n = \mathcal P$, and $\rFd (R) = n$. By Theorem \ref{fintype}, $\mathcal B _T = \{ T \}^{\perp_\infty} = (\mathcal{P}^{<\infty})^{\perp}$, whence $\mathcal P = {}^{\perp}((\mathcal{P}^{<\infty})^{\perp})$ and $\rfd (R) = \rFd (R) = n$ by \cite[6.14]{GT}.
\end{proof}    

\begin{corollary}\label{finiteglobal} Let $R$ be a right artinian ring of right global dimension $n < \infty$. Let $\mathcal S$ be a representative set of all simple modules and $T = \bigoplus_{S \in \mathcal S} E(S)$. Then $T = T_f$ is an $n$-tilting module such that $\{ T \}^{\perp_\infty} = (\mathcal{P}^{<\infty})^\perp$. 
\end{corollary} 

\begin{remark}\label{relation} \rm 
For a $1$-tilting or silting $A$-module $T$, constructions of $1$-tilting or silting modules $T^\prime$ over the triangular matrix ring $\La$ related to $T$ appear in several recent works, see e.g.\ \cite{CGR} and \cite{GH}. On the one hand, our approach is more particular, since in the next section, we will restrict ourselves to the triangular matrix $k$-algebras $\La = \tilde{A}$ introduced by Cummings in \cite{C} . Moreover, we will primarily concentrate on the particular case of the tilting right $\tilde{A}$-modules $T_f$ (that are unique up to equivalence of tilting modules). On the other hand, we will deal with $n$-tilting modules for an arbitrary $n < \omega$. Moreover, in the case when $A$ has right global dimension $n < \infty$, we will see that the tilting modules $T_f$ and $\tilde{T}_f$ over $A$ and $\tilde{A}$ have different projective dimensions, namely $n$ and $n+1$, respectively.  
\end{remark}

\medskip
\section{The Cummings construction}\label{cum}

From now on, we will restrict ourselves to particular triangular matrix algebras introduced in \cite{C}. The setting is as follows:

Let $k$ be an algebraically closed field and $A$ be a basic indecomposable finite-dimensional $k$-algebra with the Jacobson radical $J(A)$. Let $\{e_i \mid 1 \leq i \leq q\}$ be a complete basic set of primitive pairwise orthogonal idempotents of $A$. For each $1 \leq i \leq q$, let $S_i$ denote the simple right $A$-module corresponding to $e_i$, i.e., $S_i = e_iA/e_iJ(A)$. Then $e_i A e_j \subseteq J(A)$ and $\dim _k(S_i) = 1$ for all $1 \leq i \neq j \leq q$.

Let $M$ be the $q$-dimensional $k$-module with the basis $\{m_i : 1 \leq i \leq q\}$ and the right $A$-module structure defined by ${m}_i e_i = m_i$, $m_i e_j = 0$, and $m_i a = 0$ for all $a \in J(A)$ and all $1 \leq i \neq j \leq q$.

\medskip
For each $1 \leq i \leq q$, let $B_i$ denote the commutative local ring $k[x]/x^2k[x]$ with the unit $1_i = 1 + x^2k[x]$. Let $B = \prod_{i=1}^q B_i$, and for each $1 \leq i \leq q$, let $\tilde{e}_i \in B$ be such that $(\tilde{e}_i)_i = 1_i$ and $(\tilde{e}_i)_j = 0$ for all $1 \leq i \neq j \leq q$. Then $\{ \tilde{e}_i \mid 1 \leq i \leq q\}$ is a complete basic set of primitive pairwise orthogonal idempotents of $B$. For each $1 \leq i \leq q$, let $\tilde{S}_i$ denote the simple left $B$-module corresponding to $\tilde{e}_i$, i.e., $\tilde{S}_i = B\tilde{e}_i/J(A)\tilde{e}_i$.

Notice that both $B_i$ ($1 \leq i \leq q$) and $B$ are commutative quasi-Frobenius $k$-algebras of finite representation type. In particular, each $B$-module $P$ is uniquely up to an isomorphism a direct sum of a projective (= injective) $B$-module which isomorphic to a direct sum of copies of indecomposable projective $B$-modules $B\tilde{e}_i = B_i$ ($1 \leq i \leq q$), and a semisimple $B$-module which is isomorphic to a direct sum of copies of the simple $B$-modules $\tilde{S}_i$ ($1 \leq i \leq q$). 

Notice that $\dim _k(\tilde{S_i}) = 1$ and $\pd {B}{\tilde{S_i}} = \infty$ for each $1 \leq i \leq q$. Following \cite{C}, we will equip the right $A$-module $M$ defined above with the left $B$-module structure defined by $\tilde{e}_i m_i = m_i$, $\tilde{e}_j m_i = 0$, and $b m_i = 0$, for all $b \in J(B)$ and all $1 \leq i \neq j \leq q$. Then $M$ is an $n$-dimensional ($B,A$)-bimodule. Moreover, $M = \bigoplus_{1 \leq i \leq q} M_i$ where for each $1 \leq i \leq q$, $M_i = \tilde{e}_i M = M e_i$ is a one dimensional $B$-$A$-bimodule. Also, we have a left $B$-module isomorphism $M_i \cong \tilde{S}_i$ and a right $A$-module isomorphism $M_i\cong S_i$. In particular, $\Hom B{M_i}{M_j} = 0$ and $\Hom A{M_i}{M_j} = 0$ for all $1 \leq i \neq j \leq q$, and $\End {B}{M_i} \cong k$, $\End {A}{M_i} \cong k$ for each $1 \leq i \leq q$.         

\medskip 
We will denote by $\tilde{A}$ the triangular matrix $k$-algebra  
\[
\tilde{A} \;=\; \begin{pmatrix} A & 0 \\ M & B \end{pmatrix}
\]
As $\tilde{A}$ is a particular instance of the triangular matrix algebra $\La$ from Definition \ref{trimr}, all the results from Section \ref{trmr} apply. In particular, we can make use of the representation of the categories $\lmod {\tilde{A}}$ and $\rmod {\tilde{A}}$ by triples provided by Propositions \ref{requiv} and \ref{lequiv}. 

\medskip
Since  $B$ is a commutative quasi-Frobenius $k$-algebra, we have $\rfd (B) = \rFd (B) = 0$. By a classic result of Fossum, Griffith and Reiten, $\rfd (A) \leq \rfd (\tilde{A}) \leq \rfd (A) + 1$, \cite[4.21]{FGR}. 

The suprising recent result of Cummings says that the left finitistic dimensions of $\tilde{A}$ are quite different from the right ones: $\lfd (\tilde{A}) = \lFd(\tilde{A} = 0$, cf.\ \cite[3.3]{C}. 

Cummings' proof of this fact uses a criterion for $\lFd{\tilde{A}} = 0$ due to Bass \cite[6.2]{B}. We start with a direct alternative proof of the latter fact employing the representation of projective left modules over triangular matrix rings from Proposition \ref{simp-proj-inj}:

\begin{lemma}\label{dualzero} 
\begin{itemize}
\item[{(i)}]  Let $(N,P,f) \in \lmod {\tilde{A}}$. Then $P = C \oplus D$ where $C$ is a completely reducible left $B$-module, $D$ is a projective left $B$-module, $f \in \Hom {B}{M\otimes_AN}{P}$, and $\im f \subseteq C$. Moreover, $(N,P,f) = (N,C,f) \oplus (0,D,0)$, and $(0,D,0)$ is a projective left ${\tilde{A}}$-module. The left ${\tilde{A}}$-module $(N,C,f)$ is projective, iff $N$ is a projective left $A$-module and $f$ is bijective. 
\item[{(ii)}]  All left ${\tilde{A}}$-modules of finite projective dimension are projective, that is, 

$\lFd{\tilde{A}} = \lfd{\tilde{A}} = 0$.
\end{itemize}
\end{lemma}

\begin{proof} (i) The decomposition of $P$ follows from $B$ being of finite representation type with all indecomposable modules projective or simple. Since $J(B).M = 0$, the left $B$-module $M\otimes_AN$ is completely reducible, whence $\im f \subseteq C$. This yields the decomposition $(N,P,f) = (N,C,f) \oplus (0,D,0)$. The left $B$-module $(0,D,0)$ is projective by Proposition \ref{simp-proj-inj}(v). 

Let $0 \to K \overset{\mu}\to F \overset{\pi}\to N \to 0$ be a projective presentation of $N$ in $\lmod A$, i.e., a short exact sequence in $\lmod A$ with $F$ projective. Since for each $1 \leq i \leq q$, there is a short exact sequence in $\lmod B$ of the form $0 \to \tilde{S}_i \to B_i \to \tilde{S}_i \to 0$, the completely reducible left $B$-module $C$ has a projective presentation of the form $0 \to C \to G \overset{\rho} \to C \to 0$. In view of Proposition \ref{simp-proj-inj}(v), we have the following projective presentation of $(N,C,f)$ in $\lmod {\tilde{A}}$:

$$(\ast) \quad 0 \to (K,\Ker{(\rho \oplus \gamma)},f^\prime) \overset{(\mu,\iota)}\longrightarrow (0,G,0) \oplus (F, M \otimes_A F,id) \overset{(0,\rho) \oplus (\pi,\gamma)}\longrightarrow (N,C,f) \to 0$$ 

where $\gamma = (M \otimes_A \pi) f$ and there is a commutative diagram in $\lmod B$

\begin{tikzcd}[column sep=2.5em, row sep=2.2em]           
& M \otimes_A K \arrow[r, "M \otimes_A \mu"] \arrow[d, "f^\prime"]
& M \otimes_A F \arrow[r, "M \otimes_A \pi"] \arrow[d, "\sigma"]
& M \otimes_A N \arrow[r] \arrow[d, "f"]
& 0 \\
0 \arrow[r]
& \Ker{(\rho \oplus \gamma)} \arrow[r, "\subseteq"]
& G \oplus (M \otimes_A F) \arrow[r, "\rho \oplus \gamma"]
& C \arrow[r]
& 0 \\
\end{tikzcd} 

where $\sigma$ denotes the embedding of $M \otimes_A F$ on to the second component of $G \oplus (M \otimes_A F)$. 

If $(N,C,f)$ is projective, then $(\ast)$ splits. In particular, $\mu$ splits and $N$ is projective. Hence w.l.o.g., $F = N$, $K = 0$, $\pi = id_N$, $\gamma = f$, and $f^\prime = 0$. The splitting of $(\ast)$ then yields $\varphi \in \Hom BC{G \oplus (M \otimes_A N)}$ such that $\varphi (\rho \oplus f) = id$ and $f \varphi = id_{M \otimes _A N}$. The latter equality implies that $f$ is monic. 

Notice that $\Ker{(\rho \oplus f)}$ is isomorphic to the pullback of $f$ and $- \rho$ in $\lmod B$. Since $f$ is monic, we have the following commutative diagram in $\lmod B$:

$$\begin{CD}
0@>>> {\Ker{(\rho \oplus f)}}@>>> G@>>> {\coker (f)}@>>> 0\\
@. @VVV @V{- \rho}VV @| @.\\
0@>>> {M \otimes_A N} @>{f}>> C @>>> {\coker (f)}@>>> 0.\\
\end{CD}$$

so that $\Ker{(\rho \oplus f)}$ is isomorphic to the first syzygy of the completely reducible left $B$-module $\coker (f)$. If $\coker (f) \neq 0$, then  
$\Ker{(\rho \oplus f)}$ is not a projective left $B$-module, whence $(0,\Ker{(\rho \oplus f)},0)$ is not a projective left ${\tilde{A}}$-module by Proposition \ref{simp-proj-inj}(v), a contradiction. Thus $f$ is a bijection.

Conversely, if $N$ is a projective left $A$-module and $f$ is bijective, then $N \overset{\alpha}\cong \bigoplus_{1 \leq i \leq q} (A e_i)^{(\kappa_i)}$ and $C \cong \bigoplus_{1 \leq i \leq q} \tilde{S}_i^{(\kappa_i)}$ for uniquely determined cardinals $\kappa_i$ ($1 \leq i \leq q)$. Then $(N,C,f)$ is isomorphic to the projective left ${\tilde{A}}$-module $\bigoplus_{1 \leq i \leq q} X_i^{(\kappa_i)}$ for $X_i = (Ae_i,\tilde{S}_i,id_{i})$ ($1 \leq i \leq q$) via the isomorphism $(\alpha,f^{-1}(M \otimes_A \alpha)(\bigoplus_{1 \leq i \leq q} id_{i}))$.

(ii) It suffices to prove that there exists no left ${\tilde{A}}$-module $(N,C,f)$ of projective dimension $1$ such that $C$ is a completely reducible left $B$-module. 

Let $(N,C,f)$ be such a module. Then, by part (i), in the presentation $(\ast)$ above $K$ is a projective left $A$-module, $\Ker{(\rho \oplus \gamma)} = C^\prime \oplus D^\prime$ where $C^\prime$ is a completely reducible left $B$-module, $D^\prime$ a projective left $B$-module, and $f^\prime :  M \otimes_A K \to C^\prime$ is bijective. As $\Ker{(\rho \oplus \gamma)}$ contains the completely reducible submodule $C \oplus 0$, we have $C \oplus 0 \subseteq C^\prime = \im {f^\prime} \subseteq \im {\sigma} = 0 \oplus (M \otimes_A F)$. Thus, $C = 0$, and $(N,C,f) = (N,0,0)$ has a projective presentation $0 \to (K,M\otimes_A F,M \otimes_A \mu) \to (F,M \otimes_A F,id) \to (N,0,0) \to 0$. By the above, $M \otimes_A \mu$ is a bijection, so $M \otimes _A \pi = 0$ and $M \otimes _A N = 0$. Then $S_i \otimes _A N = 0$ for each $1 \leq i \leq q$, and $A \otimes _A N \cong N = 0$, a contradiction.         
\end{proof}

The fact that $\lFd{\tilde{A}} = 0$ has surprising consequences for the structure of (infinite dimensional) cotilting right $\tilde{A}$-modules. To see that, we now briefly recall the relevant facts on cotilting modules and classes from \cite[Chap.\ 15]{GT}.

\medskip
Let $R$ be a ring. An module $C$ is \emph{cotilting} if it satisfies dual conditions to those defining tilting modules: 

(C1) $C$ has finite injective dimension, (C2) $\Ext 1R{C^\kappa}{C} = 0$ for each cardinal $\kappa$, and (C3) If $I$ is an injective cogenerator of $\rmod R$, then there exists an $r \geq 0$ and an exact sequence $0 \to C_r \to \dots \to C_0 \to I \to 0$ such that $C_i \in \Prod(C)$ for all $i \leq r$, where $\Prod(C)$ denotes the class of all direct summands of direct sums of copies of the module $C$.  

Two cotilting modules $C$ and $C^\prime$ are \emph{equivalent} in case $\Prod(C) = \Prod(C^\prime)$.

For example, if $T$ is any tilting left $R$-module, then the dual module $C = D(T)$ is a cotilting (right $R$-) module, \cite[15.2]{GT}. However, in general, there exist cotilting modules that are not equivalent to duals of any tilting left $R$-modules, \cite[15.33-34]{GT}.  

A cotilting module $C$ is \emph{$n$-cotilting} if its injective dimension is $\leq n$. Each $n$-cotilting module determines a hereditary \emph{$n$-cotilting cotorsion pair} $(\mathcal A,\mathcal B)$ with $\mathcal A = {}^{\perp_\infty} \{ C \}$; the class $\mathcal A$ is the \emph{$n$-cotilting class} of $C$. 

By \cite[15.9]{GT}, there is a characterization of $n$-cotilting cotorsion pairs dual to Theorem \ref{chartilt}: a hereditary cotorsion pair in $\rmod R$, $(\mathcal A,\mathcal B)$, is $n$-cotilting, if and only if the class $\mathcal A$ is closed under direct products and the class $\mathcal B$ consists of modules of injective dimension $\leq n$. 

\medskip
Now we can state and prove 

\begin{theorem}\label{nocotilt} All cotilting right $\tilde{A}$-modules are injectice (i.e., they are injective cogenerators of 
$\rmod {\tilde{A}}$).
\end{theorem}
 
\begin{proof} Assume there exists an $n$-cotilting right $\tilde{A}$-module $C$ of injective dimension $n \geq 1$. If $n \geq 2$, then by (the dual version of) Lemma \ref{tiltshift}, there is a cotilting right $\tilde{A}$-module $C^\prime$ of injective dimension $n-1$ whose cotilting class equals ${}^{\perp_{\infty}} (\Omega^{-1}(C))$, where $\Omega^{-1}(C)$ denotes the first cosyzygy of $C$. By induction, we obtain thus a $1$-cotilting right $\tilde{A}$-module $\bar{C}$ of injective dimension $1$. By (the right hand version of) \cite[15.31]{GT}, $\bar{C}$ is equivalent to $D(T)$ for a $1$-tilting left $\tilde{A}$-module $T$. However, $T$ is projective by Lemma \ref{dualzero}(ii), so $\bar{C}$ is injective, a contradiction. 

Thus, all cotilting right $\tilde{A}$-modules are injective, Whence they are injective cogenerators in $\rmod {\tilde{A}}$.
\end{proof}

\medskip
\section{Right $\tilde{A}$-modules of finite projective dimension}\label{rfind}

We will keep the notation of Section \ref{cum}. In particular, we will represent right $\tilde{A}$-modules as the triples $(N,P,f)$ where $N \in \rmod A$, $P \in \rmod B$ and $f \in \Hom {A}{P \otimes _B M}{N}$ as in Proposition \ref{requiv}. Our goal is to understand relations between the tilting right $A$-module $T_f = T_A$ and the tilting right $\tilde{A}$-module $T_f = \tilde{T}_{\tilde{A}}$ whose existence---in the case when $\rfd{A}$ is finite---follows from Theorem \ref{findim}.  

We start with an investigation of the relations between the right $A$- and right $\tilde{A}$- modules of finite projective dimensions.  

\begin{lemma}\label{projdim} Assume that $(N,P,f)$ is a right $\tilde{A}$-module of finite projective dimension. Then $P$ is a projective right $B$-module.
\end{lemma}

\begin{proof} If $(N,P,f)$ is a projective right $\tilde{A}$-module, then $P$ is a projective right $B$-module by Proposition \ref{simp-proj-inj}(ii).  

(i) Assume that $(N,P,f)$ has projective dimension $n \geq 1$. Let $0 \to K \overset{\mu}\to F \overset{\pi}\to N \to 0$ be a projective presentation of $N$ in $\rmod A$, and $0 \to L \overset{\nu}\to G \overset{\rho}\to P \to 0$ a projective presentation of $P$ in $\rmod B$. 

In view of Proposition \ref{simp-proj-inj}(ii), we have the following non-split projective presentation of $(N,P,f)$ in $\rmod {\tilde{A}}$:

$$(\ast \ast) \quad \quad 0 \to (\Ker{(\pi \oplus g)},L,f^\prime) \overset{(\subseteq,\nu)}\longrightarrow (F,0,0) \oplus (G \otimes _B M,G,id) \overset{(\pi \oplus g,\rho)}\longrightarrow (N,P,f) \to 0$$ 

where where $g = f (\rho \otimes_B M)$, and there is a commutative diagram in $\rmod A$

\[
\begin{tikzcd}[column sep=2.5em, row sep=2.2em]      
& L \otimes_B M \arrow[r, "\nu \otimes_B M"] \arrow[d, "f^\prime"]
& G \otimes_B M \arrow[r, "\rho \otimes_B M"] \arrow[d, "\iota"]
& P \otimes_B M \arrow[r] \arrow[d, "f"]
& 0 \\
0 \arrow[r]
& \Ker{(\pi \oplus g)} \arrow[r, "\subseteq"]
& F \oplus (G \otimes_B M) \arrow[r, "\pi \oplus g"]
& N \arrow[r]
& 0 
\end{tikzcd} 
\]

where $\iota$ denotes the embedding of $G \otimes_B M$ on to the second component of $F \oplus (G \otimes_B M)$. 

(ii) If $P$ is not a projective right $B$-module, then $P$ has infinite projective dimension, and so does $L$, whence $(\Ker{(\pi \oplus g)},L,f^\prime)$ has projective dimension $n-1$, but it is not projective. Repeating the argument above for $(\Ker{(\pi \oplus g)},L,f^\prime)$ in place of $(N,P,f)$ etc., we conclude that $\pd{\tilde{A}}{(N,P,f)} = \infty$.
\end{proof}

Each projective right $\tilde{A}$-module is a (unique up to the ordering of summands and isomorphism) direct sum of indecomposable projective summands. So its structure is clear from Proposition \ref{simp-proj-inj}(ii). For modules of projective dimension $\geq 1$, we have the following characterization: 

\begin{lemma}\label{charproj} Let $(N,P,f)$ be a non-projective right $\tilde{A}$-module such that $P$ is a projective right $B$-module. Let $0 \to K \overset{\mu}\to F \overset{\pi}\to N \to 0$ be a projective presentation of $N$ in $\rmod A$. Then $(N,P,f)$ has a non-split projective presentation 

$$(\ast \ast \ast) \quad 0 \to (N^\prime,0,0) \overset{(\subseteq,0)}\longrightarrow (F,0,0) \oplus (P \otimes _B M,P,id) \overset{(\pi \oplus f,id)}\longrightarrow (N,P,f) \to 0.$$ 

where $N^\prime = \Ker{(\pi \oplus f)}$. Moreover, $\pd{\tilde{A}}{(N,P,f)} =  1 + \pd{A}{N^\prime}$. 
\end{lemma}

\begin{proof} The presentation ($\ast \ast \ast$) is a particular instance of ($\ast \ast$) from part (i) of the proof of Lemma \ref{projdim} for $L = 0$ and $G = P$. By Remark \ref{pd-id}, $\pd{A}{N^\prime} = \pd{\tilde{A}}{(N^\prime,0,0)}$, and the claim follows.   
\end{proof}

Lemmas \ref{projdim} and \ref{charproj} yield

\begin{corollary}\label{cordim} $\rFd {\tilde{A}} \leq \rFd {\tilde{A}} + 1$ and  $\rfd {\tilde{A}} \leq \rfd {\tilde{A}} + 1$.
\end{corollary}

\begin{lemma}\label{extension}
Let $(N,P,f)$ be a right $\tilde{A}$-module such that $P$ is a projective right $B$-module. Then $P$ has a decomposition $P = P^\prime \oplus P^{\prime \prime}$ that induces an exact sequence of right $\tilde{A}$-modules
$$(\dagger) \quad 0 \to (N,P^\prime,f^\prime) \overset{(id,\alpha)}\longrightarrow (N,P,f) \overset{(0,\beta)}\longrightarrow (0,P^{\prime \prime},0) \to 0$$ 
where $\alpha : P^\prime \to P$ is the split inclusion, $\beta : P \to P^{\prime \prime}$ the split projection, $\Ker f \cong P^{\prime \prime} \otimes_B M$, and $f^\prime = f \restriction (P^{\prime} \otimes_B M)$ is a monomorphism. 

Moreover, if $(N,P^\prime,f^\prime)$ is not projective, then 
$$\pd{\tilde{A}}{(N,P^\prime,f^\prime)} = \pd{A}{\coker f}.$$ 
Also, if $\Ker{f} \neq 0$, then $\pd{\tilde{A}}{(0,P^{\prime \prime},0)} = 1 + \pd{A}{\Ker f}$.
\end{lemma}

\begin{proof} Since $P$ is a projective right $B$-module, w.l.o.g., $P = \bigoplus_{1 \leq i \leq q} B_i ^{(U_i)}$ and $P \otimes_B M = \bigoplus_{1 \leq i \leq q} S_i ^{(U_i)}$ for some sets $U_i$ ($1 \leq i \leq q$). Since $P \otimes_B M$ is a completely reducible right $A$-module, $\Ker f$ is a direct summand in $P \otimes_B M$. So there exists subsets $V_i \subseteq U_i$ ($1 \leq i \leq q$) such that 
$\Ker f \oplus \bigoplus_{1 \leq i \leq q} S_i ^{(V_i)} = P \otimes_B M$. 

Let $P^\prime = \bigoplus_{1 \leq i \leq q} B_i ^{(V_i)}$ and  $P^{\prime \prime} = \bigoplus_{1 \leq i \leq q} B_i ^{(U_i \setminus V_i)}$. Then $\Ker f \oplus (P^\prime \otimes_B M) = P \otimes_B M$ and 
$\Ker f \cong P^{\prime \prime} \otimes_B M$. Moreover, $f^\prime = f \restriction (P^\prime \otimes_B M)$ is monic. 

Since the diagram 

\begin{tikzcd}[column sep=2.5em, row sep=2.2em]
0 \arrow[r] 
& P^\prime \otimes_B M \arrow[r, "\alpha \otimes_B M"] \arrow[d, "f^\prime"] 
& P \otimes _B M \arrow[r, "\beta \otimes _B M"] \arrow[d, "f"] 
& P^{\prime \prime} \otimes _B M \arrow[r] \arrow[d] 
& 0 \\
& N \arrow[r, equal] 
& N \arrow[r]  
& 0 
\end{tikzcd}

is commutative, the sequence $(\dagger)$ is exact in $\tilde{A}$. 

By Proposition \ref{simp-proj-inj}(ii), the sequence 
$$0 \to {(P^{\prime \prime} \otimes _B M,0,0)} \overset{(0,id)}\longrightarrow {(P^{\prime \prime} \otimes _B M,P^{\prime \prime},id)} \overset{(0,id)}\longrightarrow {(0,P^{\prime \prime},0)} \to 0$$
is a projective presentation of $(0,P^{\prime \prime},0)$. Since $P^{\prime \prime} \otimes _B M \cong \Ker f$, if $\Ker f \neq 0$, i.e., if $P^{\prime \prime} \neq 0$, then the presentation above does not split by Remark \ref{pd-id}, whence $\pd{\tilde{A}}{(0,P^{\prime \prime},0)} = 1 + \pd{A}{\Ker f}$.

Since $P^\prime$ is projective, if $X = (N,P^\prime,f^\prime)$ is not a projective right $\tilde{A}$-module, then Lemma \ref{charproj} applies to $X$, and yields a first syzygy of $X$ of the form $(N^\prime,0,0)$ where $N^\prime$ is the kernel of the map $\pi \oplus f^\prime : F \oplus (P^\prime \otimes_B M) \to N$. Since $f^\prime$ is monic, $N^\prime$ is isomorphic to the pull-back of $f^\prime$ and $- \pi$. Then $N^\prime$ is also isomorphic the first syzygy of the right $A$-module $\coker f^\prime = \coker f$. Thus $\pd{\tilde{A}}{(N,P^\prime,f^\prime)} = \pd{A}{\coker f}$.  
\end{proof}

The injective finitely generated right $\tilde{A}$-module $(0,B,0)$ is easily seen to be a test module for the finiteness of $\rgd{\tilde{A}}$: 

\begin{lemma}\label{test} $\pd{\tilde{A}}{(0,B,0)} = 1 + \max_{1 \leq i \leq q}\pd{A}{S_i}$. 

So $\pd{\tilde{A}}{(0,B,0)} = 1 + \rgd {A}$.
\end{lemma}

\begin{proof} $(0,B,0)$ has a non-split projective presentation $0 \to (M,0,0) \to (M,B,id) \to {(0,B,0)} \to 0$, where $M = \bigoplus_{1 \leq i \leq q} S_i$, and the claim follows by Remark \ref{pd-id}.
\end{proof}

The next lemma will be useful for checking the vanishing of Ext in $\rmod {\tilde{A}}$:

\begin{lemma}\label{vanish} Let $Q \in \rmod A$ and $(N,P,f) \in \rmod {\tilde{A}}$. Then for each $i \geq 1$, $\Ext i{\tilde{A}}{(Q,0,0)}{(N,P,f)} = 0$, if and only if $\Ext i{A}{Q}{N} = 0$.
\end{lemma} 

\begin{proof} In view of Proposition \ref{simp-proj-inj}(ii) and Remark \ref{pd-id}, the 1st syzygy $\Omega((\bar{N},0,0))$ of a module $(\bar{N},0,0) \in \rmod {\tilde{A}}$ can be taken of the form $(\Omega(\bar{N}),0,0))$ where $\Omega(\bar{N})$ is the 1st syzygy of $\bar{N}$ in $\rmod A$. Thus it suffices to prove the claim for $i = 1$, and then use dimension shifting.

Consider a short exact sequence in $\rmod A$ of the form 
$$(\Delta) \quad 0 \to N \to N^\prime \to Q \to 0.$$  
Then there is a short exact sequence of the form 
$$0 \to (N,P,f) \to (N^\prime,P,f) \to (Q,0,0) \to 0$$ 
in $\rmod {\tilde{A}}$. If it splits, then so does $(\Delta)$. This proves the only-if part.

Conversely, let $0 \to (N,P,f) \overset{(\alpha,\beta)}\longrightarrow (N^\prime,P^\prime,f^\prime) \overset{(\gamma,0)}\longrightarrow (Q,0,0) \to 0$ be a short exact sequence in $\rmod {\tilde{A}}$. Then we have the commutative diagram 

\begin{tikzcd}[column sep=2.5em, row sep=2.2em]
0 \arrow[r] 
& P \otimes_B M \arrow[r, "\beta \otimes_B M"] \arrow[d, "f"] 
& P^\prime \otimes _B M \arrow[r, "0"] \arrow[d, "f^\prime"] 
& 0 \arrow[d] \\
0 \arrow[r] 
& N \arrow[r, "\alpha"] 
& N^\prime \arrow[r, "\gamma"] 
& Q \arrow[r] 
& 0
\end{tikzcd}

If $\Ext 1{A}{Q}{N} = 0$, then the bottom row splits, so there is $\alpha^\prime : N^\prime \to N$ such that $\alpha^\prime \alpha = id$. The commutativity of the left-hand square yields $\alpha f = f^\prime (\beta \otimes_B M)$. Since $\beta$ is an isomorphisms, $\alpha f (\beta^{-1} \otimes_B M) = f^\prime$, and $\alpha ^\prime f^\prime = \alpha^\prime \alpha f (\beta^{-1} \otimes_B M) = 
f (\beta^{-1} \otimes_B M)$. So $(\alpha^\prime,\beta^{-1})$ is a spliting map for $(\alpha,\beta)$, proving the if-part.
\end{proof}

We arrive at a general relation between tilting modules in $\rmod {\tilde{A}}$ and in $\rmod A$:

\begin{theorem}\label{generalrel} Let $X = (N,P,f)$ be an $n$-tilting $\tilde{A}$-module of projective dimension $n \geq 1$. Let $N^\prime \in \rmod A$ be the module from Lemma \ref{charproj}, so $X^\prime = (N^\prime,0,0)$ is the first syzygy of $X$ in $\rmod {\tilde{A}}$. 

Then $\mathcal T ^\prime = \{ N^\prime \}^{\perp_{\infty}}$ is an $(n-1)$-tilting class in $\rmod A$, and $\mathcal T = \{ X^\prime \}^{\perp_\infty} = \{ (\bar{N},\bar{P},\bar{f})\in \rmod {\tilde{A}} \mid \bar{N} \in \mathcal T \}$ is an $(n-1)$-tilting class in $\rmod {\tilde{A}}$. 
\end{theorem}

\begin{proof} $\mathcal T$ is an $(n-1)$-tilting class in $\rmod {\tilde{A}}$ by Lemma \ref{tiltshift}. Moreover, $\mathcal T = \{ (\bar{N},\bar{P},\bar{f})\in \rmod {\tilde{A}} \mid \bar{N} \in \mathcal T ^\prime \}$ by Lemma \ref{vanish}. Lemma \ref{vanish} also shows that the class $\mathcal T ^\prime = \{ N^\prime \}^{\perp_{\infty}}$ is closed under direct sums. Since $\pd {A}{N^\prime} = n-1$ by Lemma \ref{charproj}, $\mathcal T$ is an $(n-1)$-tilting class in $\rmod A$ by Theorem \ref{chartilt}.
\end{proof}

\medskip
The results of this section already make it possible to describe in detail the following simple case of the Cummings construction:

\begin{proposition}\label{simcase} Assume that $A = B$. Then $\rgd {\tilde{A}} = \infty$, but $\rfd {\tilde{A}} = \rFd {\tilde{A}} = 0$. In particular, we can take $\tilde{T}_f = (A,0,0) \oplus (M,B,id)$. 
\end{proposition}

\begin{proof} Assume that $(N,P,f) \in \rmod {\tilde{A}}$ has finite projective dimension. Then $P$ is a projective right $B$-module by Lemma \ref{projdim}. 

If $(N,P,f)$ is not projective, then the right $A$-module $0 \neq N^\prime = \Ker{(\pi \oplus f)}$ from Lemma \ref{charproj} has finite projective dimension. Since $A = B$, $N^\prime$ is both projective and injective. Then $N^\prime$ is a direct summand in $F$, say $F = N^\prime \oplus G$, and the short exact sequence $0 \to N^\prime \subseteq F \oplus (P \otimes_B M) \overset{\pi \oplus f}\longrightarrow N \to 0$ splits. Then $(N,P,f) = (G,0,0) \oplus (P \otimes_B M,P,f)$. Since the module $(G,0,0)$ is projective by Proposition \ref{simp-proj-inj}(ii), we can w.l.o.g.\ assume that $(N,P,f) = (P \otimes_B M,P,f)$.  

If $f$ is not surjective, then there is a further decomposition $(P \otimes_B M,P,f) = (\im{f},P,f) \oplus (C,0,0)$ where $(\im{f}) \oplus C = P \otimes_B M$ and $0 \neq C$ is a completely reducible $A$-module of infinite projective dimension, in contradiction with Remark \ref{pd-id}. Thus $f$ is onto. 

By Lemma \ref{extension}, the module $(N,P^\prime,f^\prime)$ from that Lemma is projective (because $\coker {f^\prime} = \coker {f} = 0$). If $\ker {f} \neq 0$, then $\ker {f}$ is completely reducible, whence $(0,P^{\prime \prime},0)$ has infinite projective dimension by Lemma \ref{extension}. Thus $f$ is bijective, whence $(P \otimes_B M,P,f) \cong (P \otimes_B M,P,id)$ is projective, in contradiction with our assumption that $(N,P,f)$ is not projective.
\end{proof}

\medskip
\section{The case of finite global dimension of the algebra $A$}\label{fingd} 

In this section, we will restrict our attention to the setting when $\rgd {A} = n$ is finite. In this case, Corollary \ref{cordim} and Lemma \ref{test} give $\rfd {\tilde{A}} = \rFd {\tilde{A}} = n + 1$. However, $\rgd {\tilde{A}} = \infty$:

\begin{example}\label{exb} Since $B$ is commutative, the $(B,A)$-bimodule $M$ is also a right $B$-module, so we can consider the (finitely generated) right $\tilde{A}$-module $X = (0,M,0)$. We will show that $X$ has infinite projective dimension. 

There is a projective presentation of $M$ of the form $0 \to M \overset{\iota}\hookrightarrow B \to M \to 0$ with $\iota \otimes_B M = 0$. So we have the following non-split projective presentation in $\rmod {\tilde{A}}$: $0 \to (M,M,0) \to (M,B,id) \to X \to 0$. However, $(M,M,0) = (M,0,0) \oplus X$, and $\pd {\tilde{A}}{(M,0,0)} = \rgd A = n$ by (the proof of) Lemma \ref{test}.  
It follows that $\pd {\tilde{A}}{X} = \infty$.
\end{example}

By Theorem \ref{findim}, the fact that $\rfd {\tilde{A}} = n + 1$ implies that the tilting right $\tilde{A}$-module $\tilde{T}_f$ exists and has projective dimension $n+1$. Our main result gives a complete description of this module in the case when $\rgd {A} < \infty$:

\begin{theorem}\label{main} Assume $\rgd {A} = n < \infty$. Let $T = \bigoplus_{1\ \leq i \leq q} E(S_i)$ be the minimal injective cogenerator of $\rmod A$ (see Corollary \ref{finiteglobal}). Let $\iota : M \to T$ be the inclusion of the socle of $T$ into $T$, and $\tilde{T}_f = (T,B,\iota) \oplus (0,B,0)$. 

Then $\tilde{T}_f$ is a $(n+1)$-tilting right $\tilde{A}$-module such that $\{ \tilde{T}_f \}^{\perp_\infty} = (\mathcal P ^{< \infty})^{\perp}$.
\end{theorem}

\begin{proof} First, we verify conditions (T1)-(T3) from Definition \ref{deftilt} for $\tilde{T}_f$.

(T1) By Lemma \ref{test}, $\pd{\tilde{A}}{(0,B,0)} = n + 1$. Since $\iota$ is monic, the decomposition of $P = B$ from Lemma \ref{extension} is trivial. If $(T,B,\iota)$ is not projective, then Lemma \ref{extension} yields $\pd{\tilde{A}}{(T,B,\iota)} = \pd{A}{\coker f} \leq n$. Thus $\pd {\tilde{A}}{\tilde{T}_f} = n + 1$. 

(T2) The right $\tilde{A}$-module $(0,B,0)$ is injective by Proposition \ref{simp-proj-inj}(iii) and $T_f$ is finitely presented, so by \cite[6.6]{GT}, we only have to prove the vanishing of 
the Ext-groups $\Ext i{\tilde{A}}{\tilde{T}_f}{\tilde{T}_f}$ and $\Ext i{\tilde{A}}{(0,B,0)}{\tilde{T}_f}$ for each $1 \leq i < \infty$. 

By (the proof of) Lemma \ref{test}, the first syzygy of $(0,B,0)$ is $(M,0,0)$. Similarly, if $X = (T,B,\iota)$ is not projective, then by Lemma \ref{charproj}, its first syzygy is $(N^\prime,0,0)$ where $N^\prime$ is the first syzygy of the right $A$-module $T/M$. So by Lemma \ref{vanish} and Remark \ref{pd-id}, in order to prove the vanishing of the two Ext$^i$-groups for each $2 \leq i < \infty$, it suffices to show that the Ext-groups $\Ext {i-1}{A}{M}{T}$ and $\Ext {i}{A}{T/M}{T}$ vanish for all $2 \leq i < \infty$. However, this is immediate, since $T$ is an injective right $A$-module. 

Notice that $\iota = \bigoplus_{1 \leq i \leq q} \iota_i$ where $\iota_i : S_i = B_i \otimes _B M \to E(S_i)$ is the inclusion of the socle $S_i$ of $E(S_i)$ into $E(S_i)$. Hence $X = \bigoplus_{1 \leq i \leq q} T_i$ where $T_i = (E(S_i),B_i,\iota_i)$. So in order to show that $\Ext 1{\tilde{A}}{X}{X} = 0$, it suffices to prove that $\Ext 1{\tilde{A}}{T_i}{T_j} = 0$ for all $1 \leq i,j \leq q$. 

Let 

$$(\ddagger) \quad 0 \to (E(S_j),B_j,\iota_j) \overset{(\alpha,\beta)}\longrightarrow (N,P,f) \overset{(\gamma,\delta)}\longrightarrow (E(S_i),B_i,\iota_i) \to 0$$ 

be a short exact sequence in $\rmod {\tilde{A}}$. W.l.o.g., $\alpha$ and $\beta$ are split inclusions, and $\gamma$ and $\delta$ are split epimorphisms, and we have the commutative diagram 

\begin{tikzcd}[column sep=2.5em, row sep=2.2em]
0 \arrow[r] & S_j \arrow[r, "\beta \otimes_B M"] \arrow[d, "\iota_j"] & P \otimes _B M \arrow[r, "\delta \otimes _B M"] \arrow[d, "f"] & S_i \arrow[r] \arrow[d, "\iota_i"] & 0 \\
0 \arrow[r] & E(S_j) \arrow[r, "\alpha"] & N \arrow[r, "\gamma"] & E(S_i) \arrow[r] & 0
\end{tikzcd}

where $\beta \otimes_B M$ is a split inclusion. Let $C = \im {\beta} \otimes_B M$ and $D = \im {\delta^\prime} \otimes_B M$ where $\delta^\prime$ is a split monomorphism such that $\delta \delta ^\prime = id$. Then $C \oplus D = P \otimes_B M$. Since $\gamma$ vanishes on $f(C)$ while $\gamma$ is monic on $f(D)$, we infer that $f(C) \cap f(D) = 0$. By the commutativity of the diagram, $E(S_j) = E(f(C))$ and $E(f(C)) \oplus E(f(D)) = N$. Define $\beta^\prime : P \to B_j$ by $\beta^\prime \restriction \im {\beta} = \beta ^{-1}$ and $\beta^\prime \restriction \im {\delta^\prime} = 0$. Similarly, $\alpha^\prime : N \to E(S_j)$ is defined by $\alpha^\prime \restriction E(f(C)) = \alpha ^{-1}$ and $\alpha^\prime \restriction E(f(D)) = 0$. Then $\iota_j (\beta^\prime \otimes _B M) = \alpha^\prime f$, whence the sequence $(\ddagger)$ splits.

We have also the decomposition $(0,B,0) = \bigoplus_{1 \leq i \leq q} (0,B_i,0)$ in $\rmod {\tilde{A}}$. So it remains to prove that for all $1 \leq i,j \leq q$, $\Ext 1{\tilde{A}}{(0,B_i,0)}{T_j} = 0$. Let 

$$(\dagger \ddagger) \quad 0 \to (E(S_j),B_j,\iota_j) \overset{(id,\beta)}\longrightarrow (N,P,f) \overset{(0,\delta)}\longrightarrow (0,B_i,0) \to 0$$ 

be a short exact sequence in $\rmod {\tilde{A}}$. If $i \neq j$, then $(\dagger \ddagger)$ splits as $\Hom {A}{S_j}{S_i} = 0$. For $i = j$, we have the commutative diagram 

\begin{tikzcd}[column sep=2.5em, row sep=2.2em]
0 \arrow[r] & S_i \arrow[r, "\beta \otimes_B M"] \arrow[d, "\iota_j"] & P \otimes_B M = S_i \oplus S_i \arrow[d, "f"] \\
& E(S_i) \arrow[r, equal] & N 
\end{tikzcd}

Since $k$ is algebraically closed, the restriction of $f$ to the second copy of $S_i$ is a multiplication by some $x \in B_i$. We define $\beta^\prime : B_i \oplus B_i \to B_i$ as identity on the first component, and multiplication by $x$ on the second. Then $(\beta^\prime \otimes_B M) (\beta \otimes_B M) = id$ and $\iota_i (\beta^\prime \otimes_B M) = f$. Hence $(\dagger \ddagger)$ splits also in this case.     

In order to verify condition (T3), we will actually prove a stronger claim in the following Lemma \ref{strong}: \emph{If $X = (N,P,f)$ is any right $\tilde{A}$-module of finite projective dimension, then $X$ has a finite $\Add {\tilde{T}_f}$-coresolution}. Then clearly condition (T3) will hold, so $\tilde{T}_f$ is an $(n+1)$-tilting module. By Proposition \ref{tiltclasses}(ii), the class $\mathcal A _{\tilde{T}_f}$ contains all the right $\tilde{A}$-modules of finite projective dimension. Since $\mathcal A _{\tilde{T}_f}$ consists of modules of projective dimension $\leq n +1$ by Theorem \ref{chartilt}, we conclude from Theorem \ref{fintype} that $\tilde{T}_f$ is the $(n + 1)$-tilting right $\tilde{A}$-module such that $\{ \tilde{T}_f \}^{\perp_\infty} = (\mathcal P ^{< \infty})^{\perp}$.   
\end{proof}

Ir remains to prove 

\begin{lemma}\label{strong} If $X = (N,P,f)$ is any right $\tilde{A}$-module of finite projective dimension, then $X$ has a finite $\Add {\tilde{T}_f}$-coresolution.
\end{lemma}

\begin{proof} By Lemma \ref{projdim}, $P$ is a projective right $B$-module, hence $P \oplus Q = B^{(\kappa)}$ for a cardinal $\kappa$ and a right $B$-module $Q$. Since $T$ is an injective cogenerator of $\rmod A$, there are a cardinal $\lambda$ and a short exact sequence $0 \to N \overset{\alpha}\to T^{(\lambda)} \overset{\gamma}\to \bar{N} \to 0$. Consider the split exact sequence in $\rmod B$: \, $0 \to P \overset{\beta}\to B^{(\kappa)} \oplus B^{(\lambda)} \overset{\delta}\to Q \oplus B^{(\lambda)} \to 0$, where $\delta \restriction B^{(\lambda)} = id$. It induces a short exact sequence in $\rmod {\tilde{A}}$ with the middle term in $\Add {\tilde{T}_f}$ 

$$0 \to (N,P,f) \overset{(\alpha,\beta)}\longrightarrow (T^{(\lambda)},B^{(\kappa)} \oplus B^{(\lambda)},0 \oplus \iota^{(\lambda)}) \overset{(\gamma,\delta)}\longrightarrow (\bar{N},\bar{P},\bar{f}) \to 0$$ 

and the commutative diagram 

\begin{tikzcd}[column sep=2.5em, row sep=2.2em]
0 \arrow[r] 
& P \otimes_B M \arrow[r, "\beta \otimes_B M"] \arrow[d, "f"] 
& M^{(\kappa)} \oplus M^{(\lambda)} \arrow[r, "\delta \otimes _B M"] \arrow[d, "0 \oplus \iota^{(\lambda)}"] 
& \bar{P} \otimes _B M \arrow[r] \arrow[d, "\bar{f}"] 
& 0 \\
0 \arrow[r] 
& N \arrow[r, "\alpha"] 
& {T^{(\lambda)}} \arrow[r, "\gamma"]  
& {\bar{N}} \arrow[r] 
& 0
\end{tikzcd}

where $\bar{P} = Q \oplus B^{(\lambda)}$ is a projective right $B$-module. If $N$ is not injective, then $\id{A}{\bar{N}} = \id{A}{N} - 1$. Since $\rgd A = n$, repeating this procedure for $(\bar{N},\bar{P},\bar{f})$ etc. (at most $n$-times), we eventually arrive at an $\tilde{A}$-module $(N,P,f)$ of finite projective dimension such that $N$ is injective and $P$ projective. 

Consider the exact sequence $(\dagger)$ from Lemma \ref{extension}. Since its right-hand term $(0,P^{\prime \prime},0) \in \Add {\tilde{T}_f}$, we can w.l.o.g.\ assume that $f$ is a monomorphism. Indeed, the desired $\Add {\tilde{T}_f}$-coresolution of $(N,P,f)$ can then be obtained using the pushout of the first term/embedding of the $\Add {\tilde{T}_f}$-coresolution of $(N,P^\prime,f^\prime)$ and the left-hand part/embedding from $(\dagger)$, and the fact that $\Ext 1{\tilde{A}}{X}{Y} = 0$ for all $X, Y \in \Add {\tilde{T}_f}$ by Condition (T2).

However, if $f$ is a monomorphism, $N$ is injective and $P$ projective, then $N \cong \bigoplus_{1 \leq i \leq q} E(S_i)^{(U_i)}$ and $P \cong \bigoplus_{1 \leq i \leq q} B_i^{(V_i)}$ for some sets $V_i \subseteq U_i$ ($1 \leq i \leq q$). Then the $\tilde{A}$-module $(N,P,f)$ decomposes as $(N,P,f) = (I,P,f) \oplus (J,0,0)$ where $I = E(P \otimes_B M) \cong \bigoplus_{1 \leq i \leq q} E(S_i)^{(V_i)}$, $J \cong \bigoplus_{1 \leq i \leq q} E(S_i)^{(U_i \setminus V_i)}$, and $f : P \otimes_B M \to I$ is an isomorphism. Then $(I,P,f)$ is isomorphic to a direct sum of the modules of the form $(E(S_i),\tilde{e}_iB,\iota_i)$ where $1 \leq i \leq q$.

Let $\kappa = \max_{1 \leq i \leq q} \kappa_i$. Then $(I,P,f)$ is a direct summand in $(T,B,\iota)^{(\kappa)}$, so $(I,P,f) \in \Add {\tilde{T}_f}$. Finally, $(J,0,0)$ has an $\Add {\tilde{T}_f}$ coresolution $0 \to (J,0,0) \to (J,Q,\sigma) \to (0,Q,0) \to 0$ where $Q$ is a projective right $B$-module such that $Q \otimes M$ is the socle of $J$ and $\sigma$ is its embedding into $J$.  
\end{proof}

\medskip
\begin{openp}\label{op} \rm 

(i) Assume that $\rfd {A} = n$ and $T_f$ is the $n$-tilting right $A$-module such that $\{ T_{f} \}^{\perp_\infty} = (\mathcal P ^{< \infty})^{\perp}$. By Remark \ref{pd-id} and Corollary \ref{cordim}, $\rfd {\tilde{A}} = m$ where $n \leq m \leq n + 1$. 

\medskip
\emph{What is the structure of the $m$-tilting right $\tilde{A}$-module $\tilde{T}_f$ such that $\{ \tilde{T}_{f} \}^{\perp_\infty} = (\mathcal P ^{< \infty})^{\perp}$? Can $\tilde{T}_f$ be expressed simply in terms of $T_f$?}

\medskip 
In the general case, some information is contained in Theorem \ref{generalrel}. In the particular case when $A = B$, we can take $T_f = A$, and by Proposition \ref{simcase} $\tilde{T}_f = (T_f,0,0) \oplus (M,B,id)$. In this case, both $T_f$ and $\tilde{T}_f$ are progenerators, in $\rmod A$ and $\rmod {\tilde{A}}$, respectively.
 
In the case when $\rgd{A} = n < \infty$, we can take $T_f = \bigoplus_{1 \leq i \leq q} E(S_i)$, and by Theorem \ref{main}, $\tilde{T}_f = (T_f,M,\iota) \oplus (0,B,0)$. So $T_f$ has projective dimension $n$, while $\tilde{T}_f$ has projective dimension $n+1$. 

\medskip
(ii) Assume the setting of Theorem \ref{generalrel}, so let $n \geq 1$ and $X$ be the $n$-tilting right $\tilde{A}$-module such that $\{ X \}^{\perp_\infty} = (\mathcal P ^{< \infty})^{\perp}$ in $\rmod {\tilde{A}}$. 

\emph{Is the ($n-1$)-tilting class $\mathcal T ^\prime = \{ N^\prime \}^{\perp_\infty}$ in $\rmod A$ equal to $(\mathcal P ^{< \infty})^{\perp}$?} 

\medskip
This is true when $n = \rgd{A} < \infty$, since then by Theorem \ref{main}, 
$$X^\prime = \Omega_{\tilde{A}}(\tilde{T}_f) = (\Omega_A(T_f/\mbox{Soc}(T_f)),0,0) \oplus (M,0,0),$$ 
and $M ^{\perp} = \mathcal I _0$ is the $n$-tilting class of all injective right $A$-modules which is equal to $(\rfmod {A})^{\perp}$. Notice however that neither $X^\prime \in \rmod {\tilde{A}}$ nor $N^\prime = \Omega_A(T_f/\mbox{Soc}(T_f)) \oplus M \in \rmod A$ are tilting modules, as they both fail condition (T2).    
\end{openp}

\end{document}